\documentclass[11pt,english]{article}
\usepackage{amsfonts}
\usepackage{amsmath}
\usepackage{amssymb}
\usepackage{amscd}
\usepackage{amsthm}
\numberwithin{equation}{section}
\allowdisplaybreaks
\newtheorem{theorem}{Theorem}[section]
\newtheorem{proposition}{Proposition}[section]

\newtheorem{lemma} {Lemma}[section]
\newtheorem{corollary}{Corollary}[section]

\newtheorem{remark}{Remark}[section]

\newtheorem{definition} {{Definition}}[section]

\def\R{{\mathbb R}}
\newcommand{\be} {\begin{equation}}
\newcommand{\ee} {\end{equation}}
\newcommand{\bea} {\begin{eqnarray}}
\newcommand{\eea} {\end{eqnarray}}
\newcommand{\Bea} {\begin{eqnarray*}}
\newcommand{\Eea} {\end{eqnarray*}}

\begin{document}
\title{\bf On the fourth-order Leray-Lions problem with indefinite weight and nonstandard growth conditions}
\date{}
\maketitle
\begingroup\small
\vspace{-1.5cm}
\begin{center}
{\large \bf
 K. Kefi$^{a}$, N. Irzi$^{b},$ M.M. Al-Shomrani$^{c},$ D.D. Repov\v{s}$^{d, e}$
}
\vskip 1.0cm
\end{center}
\endgroup
\begin{abstract}
We prove the existence of at least three weak solutions for the fourth-order problem with indefinite weight involving the Leray-Lions operator with nonstandard growth conditions. The proof of our main result uses variational methods
and
 the critical theorem of Bonanno and Marano (Appl. Anal. 89 (2010), 1-10).
\end{abstract}

\section{Introduction }
\hspace{0.5cm}
In this paper, we shall show the existence of three weak solutions for the following interesting problem
\begin{equation}\label{P}
\left\{
\begin{array}{ll}
\Delta\left(a(x, \Delta u)\right) = \lambda V(x)|u|^{q(x)-2}u  &\hbox{ in} \   \Omega,\\
\\[-5mm]
u=\Delta u=0 & \hbox{on}\  \partial{\Omega},
\end{array}
\right.
\end{equation}
where $\Omega$ is a bounded domain in $\R^{N}$ ($N\geq2$) with a
smooth boundary $\partial\Omega$, $\lambda>0$ is a parameter, $V$ is a function in a generalized Lebesgue space $L^{s(x)}(\Omega),$ functions $p, q, s \in C(\overline\Omega)$ satisfy the inequalities 
$$\displaystyle 1<\min_{x\in \overline{\Omega}}q(x)\leq \max_{x\in \overline{\Omega}}q(x)<\min_{x\in \overline{\Omega}}p(x)\leq \max_{x\in \overline{\Omega}}p(x)\leq \frac{N}{2}<s(x) \ \hbox{for all} \  x\in\Omega,$$ and  $\Delta\left(a(x, \Delta u)\right)$ is the Leray-Lions operator  of the fourth-order, where  $a$ is a Carath\'{e}odory function satisfying  some  suitable supplementary conditions.
For more details about this kind of operators the reader is referred to {\sc Boureanu} \cite{bour} and {\sc Leray-Lions} \cite{Leray} (and
the references therein).

Note that the study of this type of operators is very active in
several
 fields,
 e.g. in electrorheological fluids ({\sc R\r{u}\v{z}i\v{c}ka} \cite{ruzi}), elasticity ({\sc Zhikov} \cite{zh2}), stationary thermorheological viscous flows of non-Newtonian fluids ({\sc Rajagopal-R\r{u}\v{z}i\v{c}ka}
 \cite{non}), image processing ({\sc Chen-Levine-Rao} \cite{R00}), and  mathematical description of the processes filtration of barotropic gas through a porous medium  ({\sc Antontsev-Shmarev} \cite{AN}).

Similar problems have been studied before by various authors, see e.g. recent papers of {\sc Afrouzi-Chung-Mirzapour} \cite{CAM}, {\sc Kefi-R\u{a}dulescu} \cite{Khaled}, {\sc Kong} \cite{Kong,Kong1},  and {\sc Chung-Ho} \cite{CH}. In particular,  {\sc Kefi} \cite{K}  studied the following problem
\begin{equation}\label{p}
\left\{
\begin{array}{ll}
-{\rm div}(|\nabla u|^{p(x)-2}u) = \lambda V(x)|u|^{q(x)-2}u  &\hbox{ in} \   \Omega,\\[1mm]
  u=0 & \hbox{on}\  \partial{\Omega}.
\end{array}
\right.
\end{equation}
Under the condition in problem \eqref{P}, he has shown that problem \eqref{p} has a continuous spectrum and his main argument was the Ekeland variational principal.

Before introducing our main result, we define
   $$C_{+}(\overline{\Omega}):=\{h \mid h\in C(\overline{\Omega}), h(x)>1,\ \mbox{for all}\;\; x\in \overline{\Omega}\},$$ and
  for $\eta>0$, $h\in C_+(\overline{\Omega})$, we set $$h^{-}:=\inf_{x\in\Omega}h(x),\ \ h^{+}:=\sup_{x\in\Omega}h(x)$$ and
$$[\eta]^h:=\sup\{\eta^{h^-},\eta^{h^+}\}, \ \ [\eta]_h:=\inf\{\eta^{h^-},\eta^{h^+}\}.$$
\begin{remark}
It is easy to verify that the following holds
$$[\eta]^{\frac{1}{h}}=\sup\{\eta^{\frac{1}{h^+}}, \eta^{\frac{1}{h^-}}\}, \
[\eta]_{\frac{1}{h}}=\inf\{\eta^{\frac{1}{h^+}},  \eta^{\frac{1}{h^-}}\}.$$
\end{remark}

 We denote $$\delta(x):=\sup\{\delta>0 \mid B(x,\delta)\subseteq\Omega, \ \mbox{for all}\ x\in\Omega\},$$ 
 where $B$ is the ball  of radius $\delta$
 centered at $x$. One can prove that there exists $x_0\in\Omega$ such that $B(x_0,D)\subseteq\Omega$,
where $D:=\displaystyle\sup_{x\in\Omega}\delta(x).$

Throughout this paper, we shall need the following hypotheses:
\begin{enumerate}
\item[$\bf{(H_1)}$] $a:\overline{\Omega}\times\R\rightarrow\R$ is a Carath\'{e}odory function such that $a(x, 0)=0,$ for a.e. $x\in\Omega.$
\item[$\bf{(H_2)}$] There exist  $c_1>0$ and a nonnegative function $\alpha\in L^{\frac{p(x)}{p(x)-1}}(\Omega)$ such that
$$|a(x, t)|\leq c_1 (\alpha(x)+|t|^{p(x)-1}),  \ \mbox{for a.e.} \; x\in\Omega \ \mbox{and all}\ t\in \R.$$
\item[$\bf{(H_3)}$] The following inequality holds
$$(a(x, t)-a(x, s))(t-s)\geq 0,\;\;\mbox{ for a.e.}\; x\in\Omega, \;\;\mbox{and all}\;\; s, t\in\R$$
 with equality if and only if $s=t$.
\item[$\bf{(H_4)}$] The following inequality holds
$$|t|^{p(x)}\leq \min\{ a(x, t)t,p(x)A(x, t)\},\;\mbox{ for a.e.} x\in\Omega \;\;\mbox{and all}\;\; s, t\in\R,$$
where $A :\overline{\Omega}\times\R\rightarrow\R$ represents the antiderivative of $a,$ that is,
$$A(x, t):=\displaystyle\int_{0}^{t}a(x, s)ds.$$
\item[$\bf{(H_5)}$] Assume that $V\in L^{s(x)}(\Omega)$ satisfies the following
$$V(x):=\left\{
\begin{array}{l}
\leq 0,\hspace{0.8cm} \mbox{ for}\;   x\in \Omega\setminus B(x_0,D),\\
\geq v_0, \hspace{0.7cm} \mbox{ for}\;  x\in B(x_0,\frac{D}{2}),\\
> 0, \hspace{0.8cm} \mbox{ for } \;  x\in B(x_0,D)\setminus B(x_0,\frac{D}{2}),
\end{array}
\right.
$$
where $B(x_0,D)$ is the ball of radius $D$
centered at $x_0$ and  $v_0$ is a positive constant.
\end{enumerate}

\noindent
\begin{remark}\label{c3}
\begin{enumerate}
We note the following facts:
 \item[$(1)$] $A(x, t)$ is a $C^{1}$-Carath\'{e}odory function, i.e., for every $t\in\R,$ $A(., t):\Omega\rightarrow\R$ is measurable
and $A(x, .)\in C^{1}(\R),$ for a.e. $x\in\Omega.$
\item[$(2)$] By hypothesis $(H_2),$ there exists a constant $c_3$ such that
$$|A(x, t)|\leq c_3\left(\alpha(x)|t|+|t|^{p(x)}\right), \  \mbox{for a.e.}\ x\in\Omega \;\;\mbox{and all}\;\; t\in \R.$$
\end{enumerate}
\end{remark}
In the sequel, let $$L:= w(D^N-(\frac{D}{2})^N),  \  w:=\displaystyle \frac{\pi^{\frac{N}{2}}}{\frac{N}{2}\Gamma(\frac{N}{2})},$$
where $\Gamma$ denotes the Euler function. Furthermore, let $k>0$ be the best constant for which  the inequality \eqref{k} below holds.
  The main result of this paper now reads as follows.
\begin{theorem}\label{principal}
Assume that hypotheses ${\bf(H_1})-({\bf H_5)}$ are fulfilled and that there exist $r>0$ and $d>0$ such that
\begin{equation}\label{r4}
r<\frac{1}{p^+}\Big[\frac{2d(N-1)}{D^2}\Big]_{p}L,
\end{equation}
and
\begin{eqnarray}
 \overline{w}_r &:=&\frac{1}{r}\Big\{\frac{{p^+}^{\frac{q^+}{p^-}}}{q^-}[k]^q|V|_{s(x)}\big[ [r]^{\frac{1}{p}}\big]^q\Big\}\\
 &<&~\gamma_d:=\frac{v_{0}[d]_q}{c_{3}(2^{N}-1)\Big(|\alpha|_{\frac{p(x)}{p(x)-1}}\frac{4d(N-1)}{D^2}L^{\frac{1}{p^+}-1}+\Big[\frac{4d(N-1)}{D^2}\Big]^p\Big)}.\nonumber
\end{eqnarray}
Then for every $\lambda\in\overline{\Lambda}_r:=\displaystyle \Big(\frac{1}{\gamma_d}, \frac{1}{\overline{w}_r}\Big),$ problem \eqref{P} admits at least three weak solutions.
\end{theorem}
\begin{remark}
If we set $ r=1,$ then conditions of Theorem \ref{principal}  read as follows:  
There exists $d>0$ such that
$$p^+<\Big[\frac{2d(N-1)}{D^2}\Big]_{p}L$$ and
\begin{eqnarray}
 \overline{w}_1 :=\Big\{\frac{{p^+}^{\frac{q^+}{p^-}}}{q^-}[k]^q|V|_{s(x)}\Big\}< \gamma_d
 \end{eqnarray}
\end{remark}

\begin{remark}\label{ope}
We are interested in the Leray-Lions type operators because they are quite general. Indeed, consider
\begin{equation}\label{a}
a(x,t):=\theta(x)|t|^{p(x)-2}t,
\end{equation}
where $p\in C_+(\overline{\Omega}),\ p^+<+\infty,$ and choose $\theta\in L^{\infty}(\Omega)$ such that there exists $\theta_0>0$ with $\theta(x)\geq\theta_0>0,$ for a.e. $x\in\Omega$. One can then see that \eqref{a} satisfies hypotheses $\bf{(H_1)-(H_4)}$ and we arrive at the following operator

$$\Delta\big(\theta(.)|\Delta|^{p(.)-2} \Delta u\big).$$
Note that when $\theta\equiv 1$, we get the well-known  $p(x)$-biharmonic operator $\Delta_{p(.)}^2(u),$ see {\sc Kefi-R\u{a}dulescu} \cite{Khaled}.
Moreover, we can make the  choice $$a(x,t):=\theta(x)(1+|t|^2)^{\frac{p(x)}{p(x)-2}}t,$$
and obtain the following operator $$\Delta\Big(\theta(.)(1+|\Delta u|^2)^{\frac{p(.)}{p(.)-2}}\Delta u\Big),$$
where $p$ and $\theta$ are as in \eqref{a}.
\end{remark}
In the sequel, define $a(x,t)$ as in \eqref{a} with $\theta\equiv 1.$ Then problem \eqref{P} becomes

  \begin{equation}\label{P1}
\left\{
\begin{array}{lll}
\Delta\big(|\Delta u|^{p(x)-2} \Delta u\big) = \lambda V(x)|u|^{q(x)-2}u  &\hbox{ in} \   \Omega,\\[1mm]
 u=\Delta u=0 & \hbox{on}\  \partial{\Omega},
\end{array}
\right.
\end{equation}
and we obtain the following result.
\begin{corollary}
Assume that there exist $r,d>0$ such that
\begin{equation}\label{r}
r<\frac{1}{p^+}\Big[\frac{2d(N-1)}{D^2}\Big]_{p}L,
\end{equation}
and
\begin{eqnarray}
 \overline{w}_r < \gamma_d:=\frac{v_{0}[d]_q}{c_{3}(2^{N}-1)\Big[\frac{4d(N-1)}{D^2}\Big]^p}
 \end{eqnarray}
Then for every $$\lambda\in\overline{\Lambda}_r:=\displaystyle \Big(\frac{1}{\gamma_d}, \frac{1}{\overline{w}_r}\Big),$$ problem \eqref{P1} admits at least three weak solutions.
\end{corollary}

This paper is organized as follows: in Section \ref{sec2}, we give some preliminaries and necessary background results on the Sobolev spaces with variable exponents, whereas Section \ref{sec3} is devoted to the proof of our main result.
\section{Preliminaries and Background}\label{sec2}

In this section, we recall some definitions and basic properties of variable exponent Sobolev spaces. For a deeper treatment of these spaces, we
refer the reader to {\sc Fan-Zhao} \cite{fan},
 {\sc R\u{a}dulescu} \cite{rad}, and {\sc R\u{a}dulescu-Repov\v{s}} \cite{RR}, and for the other background material to 
{\sc Papageorgiou-R\u{a}dulescu-Repov\v{s}} \cite{PRR}.\\

 Let $p\in C_{+}(\overline{\Omega})$ be such that
\begin{equation*}\label{e2.2}
1<p^{-}:=\displaystyle\min_{x\in\overline{\Omega}}p(x)\leq
p^{+}:=\displaystyle\max_{x\in\overline{\Omega}}p(x)<+\infty.\end{equation*}
We define the Lebesgue space with variable exponent as follows  $$L^{p(x)}(\Omega):=\{u \mid u:\Omega\rightarrow \mathbb{R}\;\;\mbox{is  measurable}, \int_{\Omega}|u(x)|^{p(x)}dx<\infty\},$$
which is  equipped with the  so-called Luxemburg norm
$$|u|_{p(x)}:=\inf\left\{\mu>0 \mid \int_{\Omega}\Big|\frac{u(x)}{\mu}\Big|^{p(x)}dx\leq 1\right\}.$$

Variable exponent Lebesgue spaces are like classical Lebesgue spaces in many respects: they
are Banach spaces and are reflexive if and only if $1 < p^{-}\leq
q^{+} < \infty$.  Moreover, the inclusion between Lebesgue spaces is generalized naturally:
 if $q_{1}, q_{2}$ are such that
$p_{1}(x) \leq p_{2}(x),$ a.e. $x \in \Omega$, then there exists a continuous embedding
\begin{center}$L^{p_{2}}(x)(\Omega)\hookrightarrow L^{p_{1}}(x)(\Omega)$.
\end{center}
For $u \in L^{p(x)}(\Omega)$ and $v \in L^{p'(x)}(\Omega)$, the H\"older inequality holds
\begin{equation}\label{ing}
\Big|\int_{\Omega}u v dx\Big|\leq \Big(\frac{1}{p^{-}}+\frac{1}{(p')^{-}}\Big)|u|_{p(x)}|v|_{p'(x)},
\end{equation}
where $\displaystyle \frac{1}{p(x)} + \frac{1}{p'(x)}=1$.

The modular on the space $L^{p(x)}(\Omega)$ is the map $\rho_{p(x)}:L^{p(x)}(\Omega)\rightarrow \mathbb{R}$ defined by $$\rho_{p(x)}(u):=\int_{\Omega}|u|^{p(x)}dx.$$

For any positive integer $m,$ we define the Sobolev space with variable exponents as follows:
\begin{eqnarray*} W^{m,p(x)}(\Omega):=\left\{u\in L^{p(x)}(\Omega)\mid D^{\alpha}u\in L^{p(x)}(\Omega), |\alpha|\leq m \right\},
\end{eqnarray*} where $\alpha := \left(\alpha_1, \alpha_2, . . . , \alpha_N\right)$ is a multi-index and
 $$\displaystyle|\alpha|:=\sum_{i=1}^{N}\alpha_i, \
 D^{\alpha}u:=\frac{\partial^{|\alpha|}u}{\partial^{\alpha_1}x_1....\partial^{\alpha_N}x_N}.$$
Then $W^{m,p(x)}(\Omega)$ is a separable and reflexive Banach space equipped with the norm
$$\|u\|_{m,p(x)}:=\sum_{|\alpha|\leq m}|D^{\alpha}u|_{p(x)}.$$

The space $W_{0}^{m,p(x)}(\Omega)$ is the closure of $C_{0}^{\infty}(\Omega)$ in $W^{m,p(x)}(\Omega)$.
It's well-known that both $W^{2,p(x)}(\Omega)$ and $W_{0}^{1,p(x)}(\Omega)$ are separable and reflexive Banach spaces. It follows that
$$X:=W^{2,p(x)}(\Omega)\cap W_{0}^{1,p(x)}(\Omega),$$
is also a separable and reflexive Banach space, when equipped with the norm
$$\|u\|_{X}:=\|u\|_{W^{2,p(x)}(\Omega)}+\|u\|_{W_{0}^{1,p(x)}(\Omega)}.$$

Let $$\|u\|:=\inf\left\{ \mu>0 \mid \int_{\Omega}\Big|\frac{\Delta u}{\mu}\Big|^{p(x)} dx\leq 1\right\},$$
represent a norm which is equivalent to $\|.\|_{X}$ on $X$ (see {\sc El Amrouss-Ourraoui} \cite[Remark 2.1]{elo}). Therefore in what follows, we shall consider the normed space $\left(X, \|.\|\right).$

The modular on the space $X$ is the map $\rho_{p(x)}:X\rightarrow \mathbb{R}$ defined by
$$\rho_{p(x)}(u):=\int_{\Omega}|\Delta u|^{p(x)}dx.$$
This mapping satisfies some useful properties and we cite some below.
\begin{lemma}({\sc El Amrous-Moradi-Moussaoui} \cite{ela}) \label{2.2}
For every $u, u_n \in W^{2, p(.)}(\Omega),$ the following statements hold:
\begin{enumerate}
\item[$(1)$] $\|u\|<1 \;(\mbox{resp.} =1, >1) \Longleftrightarrow \rho_{p(x)}(u)<1 \;(\mbox{resp.} =1, >1);$\\
\item[$(2)$] $[\|u\|]_p:=\displaystyle\min\{\|u\|^{p^{-}}, \|u\|^{p^{+}}\}\leq \rho_{p(x)}\leq\displaystyle\max\{\|u\|^{p^{-}}, \|u\|^{p^{+}}\}:=[\|u\|]^p;$\\
 \item [$(3)$] $\|u_n\|\rightarrow0\;(\mbox{resp.} \rightarrow\infty)\Leftrightarrow \rho_{p(x)}(u_n)\rightarrow0\;(\mbox{resp.} \rightarrow\infty).$
\end{enumerate}
\end{lemma}
\begin{proposition}\label{pr}({\sc Edmunds-Rakosnik} \cite{R7}) Let $p$ and $q$ be  measurable functions such
that $p\in L^{\infty}(\Omega)$, and $1\leq p(x)q(x)\leq\infty$, for a.e. $x\in\Omega$. Let $u\in L^{q(x)}(\Omega)$, $u\neq0$. Then
$$[|u|_{p(x)q(x)}]_p:=\min\{|u|^{p^{+}}_{p(x)q(x)},|u|^{p^{-}}_{p(x)q(x)}\}
\leq ||u|^{p(x)}|_{q(x)}$$ $$
\leq
[|u|_{p(x)q(x)}]^p:=
\max\{|u|^{p^{-}}_{p(x)q(x)}, |u|^{p^{+}}_{p(x)q(x)}\}.$$
\end{proposition}
We recall that the critical Sobolev exponent is defined as follows:
$$
p^*(x):=\left\{
\begin{array}{l}
\displaystyle\frac{Np(x)}{N-2p(x)},\quad \mbox{if}\;\; p(x)<\frac{N}{2},\\
\displaystyle+\infty, \hspace{0.8cm}\quad\quad \mbox{if}\;\; p(x)\geq \frac{N}{2}. \\
\end{array}
\right.
$$
\begin{remark}\label{embedding}({\sc Kefi} \cite{K})
  Denote the conjugate exponent of the function $s(x)$ by $s'(x)$ and set $\beta(x):=\displaystyle \frac{s(x)q(x)}{s(x)-q(x)}.$ Then there exist  compact and continuous embeddings $X \hookrightarrow L^{s'(x)q(x)}(\Omega)$ and $X \hookrightarrow L^{\beta(x)}(\Omega)$
  and the best constant $k>0$ such that
  \begin{eqnarray}\label{k}
 |u|_{s'(x)q(x)}\leq k\|u\|.
 \end{eqnarray}
\end{remark}
In order to formulate the variational approach to problem $(\ref{P}),$ let us recall the definition of a weak solution for our problem.
\begin{definition} We say that $u\in X\backslash\{0\}$ is a weak solution of problem $(\ref{P})$ if $\Delta u=0$ on $\partial{\Omega}$ and
$$ \int_{\Omega}a(x, \Delta u)\Delta vdx-\lambda\int_{\Omega}V(x)|u|^{q(x)-2}uvdx=0,\;\;\;\mbox{ for all} \;\;v\in X.$$
\end{definition}
We state the following proposition which will be needed in Section \ref{sec3}.
\begin{proposition}$(${\sc Gasi\'nski-Papageorgiou} \cite{GP}$)$\label{compact}
If $X$ is a reflexive Banach space, $Y$ is a Banach space, $Z\subset X$ is nonempty, closed and convex
subset, and $J:Z\rightarrow Y$ is completely continuous, then $J$ is compact.
\end{proposition}
Our main tool will be the following critical theorem {\sc Bonanno-Marano} \cite{bonano}, which  we restate in a more convenient form.
\begin{theorem}\label{bonano}$(${\sc Bonanno-Marano} [\cite{bonano}, Theorem 3.6]$)$ Let X be a reflexive real Banach space and
$\Phi:X\rightarrow\mathbb{R}$ a coercive, continuously G\^{a}teaux differentiable and sequentially weakly lower semicontinuous functional whose G\^{a}teaux derivative admits
a continuous inverse on $X$.
 Let $\Psi:X\rightarrow \mathbb{R}$ be a continuously G\^{a}teaux differentiable
functional whose G\^{a}teaux derivative is compact such that
$$(a_0)\quad\inf_{x\in X}\Phi(x)=\Phi(0)=\Psi(0)=0.$$
Assume that there exist $r > 0$ and $\overline{x}\in X$, with $r < \Phi(\overline{x})$, such that:
$$(a_1)\quad\frac{\displaystyle\sup_{\Phi(x)\leq r}\Psi(x)}{r}<\frac{\Psi(\overline{x})}{\Phi(\overline{x})};$$
$$(a_2)\quad
\mbox{for each}\;
 \lambda\in\Lambda_r:=\Big(\frac{\Phi(\overline{x})}{\Psi(\overline{x})}, \frac{r}{\displaystyle\sup_{\Phi(x)\leq r}\Psi(x)}\Big),
\mbox{the functional}\;
 \Phi-\lambda \Psi
\;
\hbox{is coercive}.$$
Then for each $\lambda\in\Lambda_r$, the functional $\Phi-\lambda \Psi$ has at least three distinct critical points in $X$.
\end{theorem}

  \section{Proof of the Main Result}\label{sec3}

  In this section, we present the proof of Theorem \ref{principal}. To begin, let us denote
$$\Psi(u):=\int_{\Omega}\frac{1}{q(x)}V(x)|u|^{q(x)}dx.$$

The Euler-Lagrange functional corresponding to problem $(\ref{P})$ is then defined by $I_{\lambda}: X \rightarrow\mathbb{R},$
 $$I_{\lambda}(u):=\phi(u)- \lambda\Psi(u),\,\,\mbox{for all}\;u\in X,$$
 where $$\Phi(u):=\displaystyle\int_{\Omega}A(x,\Delta u)dx.$$

 It is clear that condition $(a_0)$ in Theorem \ref{bonano} is fulfilled, and by virtue of Proposition \ref{pr}, $\Psi$ is well-defined since we have for all $u\in X,$
\begin{eqnarray*}
|\Psi(u)|\leq \frac{1}{q^-}\int_{\Omega}|V(x)||u|^{q(x)}dx \leq \frac{1}{q^-}|V(x)|_{s(x)}||u|^{q(x)}|_{s'(x)}
\leq \frac{1}{q^-}|V(x)|_{s(x)}[|u|_{s'(x)q(x)}]^q.
\end{eqnarray*}

 Moreover, by inequality \eqref{k} in Remark \ref{embedding}, one has
  $$|\Psi(u)|\leq  \displaystyle\frac{1}{q^-}|V(x)|_{s(x)}|[k\|u\|]^q,$$ therefore $\Psi$ is indeed well-defined. We shall also need the following lemma.
\begin{lemma}\label{lm}
$(i)$ The functional $\Phi$ is a coercive, continuously G\^{a}teaux differentiable and sequentially weakly lower semicontinuous functional $(${\sc Boureanu} \cite{bour}$)$ whose G\^{a}teaux derivative admits a continuous inverse on $X$.\\
$(ii)$ The functional $\Psi$ 
 is a a continuously G\^{a}teaux differentiable functional whose G\^{a}teaux derivative is compact. 
\end{lemma}
\begin{proof}
The proof splits into two parts:
\begin{enumerate}
\item[(i)] It is clear from Lemma \ref{2.2} and hypothesis $(\bf{H_4})$ that for every $u\in X$ such that $\|u\| >1,$ one has
\begin{eqnarray}\label{coe}
\Phi(u)\geq\int_{\Omega}\frac{1}{p(x)}|\Delta u|^{p(x)}dx
\geq \frac{1}{p^+}\rho_{p(x)}(u)
\geq  \frac{1}{p^+} \|u\|^{p^-},
\end{eqnarray}
and thus $\Phi$ is coercive.

For the rest of the proof, we will use the same argument as in the proof of
 {\sc Ho-Sim}~\cite[Lemma 3.2]{HS}.
 First, we
 shall show that $\Phi'$
is strictly monotone. Using $({\bf{H_3}})$ and integrating over $\Omega$, we obtain for
all $u, v\in X$ with $u\neq v,$ 
$$ 0<\displaystyle\int_{\Omega}(a(x,\Delta u)-a(x,\Delta v))(\Delta u-\Delta v)dx=<\Phi'(u)-\Phi'(v), u-v>,$$
which means that $\Phi'$
is strictly monotone. 

Note that
the strict monotonicity of $\Phi'$ implies that $\Phi'$ is an injection.
From the assertion ${(\bf{H_4})}$ it is clear that for any $u\in X$ with $\|u\|>1$, one has
$$\frac{<\Phi'(u),u>}{\|u\|}\geq \frac{\|u\|^{p^{-}}}{\|u\|}=\|u\|^{p^{-}-1},$$
and thus $\Phi'$ is coercive. Therefore it is a surjection in view of Minty-Browder Theorem for reflexive Banach space (cf. {\sc Zeidler} \cite{Zei}), so $\Phi'$ has a bounded inverse mapping $(\Phi')^{-1}:X^*\rightarrow X$.\\
Let $f_n\rightarrow f$ as $n\rightarrow +\infty$ in $X^*$ and set $u_n=(\Phi')^{-1}(f_n)$, $u=(\Phi')^{-1}(f)$. Then the boundedness of $(\Phi')^{-1}$ and $\{f_n\}$ imply that $\{u_n\}$ is bounded. Without loss of generality, we can assume that there exists a subsequence, again denoted by ${u_n},$ and $\tilde u$ such that $u_n\rightharpoonup \tilde u$ (weakly) in $X,$ which implies
$$|<f_n-f,u_n-\tilde u>|\leq \|f_n-f\|_{X^*}\|u_n-\tilde u\|.$$
We can now infer that $$\displaystyle\lim_{n\rightarrow+\infty}<\Phi'(u_n),u_n-\tilde u>=\lim_{n\rightarrow+\infty}<f_n,u_n-\tilde u>=\lim_{n\rightarrow+\infty}<f_n-f,u_n-\tilde u>=0,$$
which implies that $$\displaystyle\lim_{n\rightarrow+\infty}\int_{\Omega}a(x,\Delta u_n)(\Delta u_n-\Delta \tilde u)dx=0.$$
By invoking {\sc Boureanu}\cite[Theorem 3.2]{bour}, one can conclude that $u_n\rightarrow \tilde u$ (strongly) as $n\rightarrow +\infty$ in $X$. This yields $f_n=\Phi'(u_n)\rightarrow\Phi'(\tilde u)$ and thus $f=\Phi'(\tilde u)$, by the injectivity of $\Phi'$, we obtain $u=\tilde u$ and hence $(\Phi')^{-1}(f_n)\rightarrow(\Phi')^{-1}(f)$ and the proof of 
Lemma\ref{lm}  is thus completed.

\item[(ii)] Next, we show that $\Psi'(u)$ is compact. Let $v_n\rightharpoonup v$ in $X.$ Then
 \begin{eqnarray*}
 |<\Psi'(u),v_n>|&-&|<\Psi'(u),v>| 
 \leq \int_{\Omega}|V(x)| |u|^{q(x)-1}|v_n-v| dx\\
&\leq& |V(x)|_{s(x)}||u|^{q(x)-1}|_{\frac{q(x)}{q(x)-1}}|v_n-v|_{\beta(x)}.
\end{eqnarray*}
As a consequence of Remark \ref{embedding} and due to the compact embedding $X\hookrightarrow L^{\beta(x)}(\Omega)$, we have $|<\Psi'(u),v_n>|\rightarrow|<\Psi'(u),v>|,$ as $n\rightarrow +\infty$. This means that $\Psi'(u)$ is completely continuous. So, by  Proposition \ref{compact}, $\Psi'$ is indeed compact. \end{enumerate} \end{proof}
{\bf Proof of Theorem \ref{principal}}. As we have observed above, the functionals $\Phi$ and $\Psi$ satisfy the regularity assumptions of Theorem \ref{bonano}.
Now, let $v_d\in X$ be the function defined by
\begin{equation*}
v_d:=\left\{
\begin{array}{ccc}
&0,   & \; \; \; \; \; \mbox{if}\ \  x\in\Omega\setminus B(x_0,D),\\
&  \displaystyle \frac{2d}{D}(D-|x-x_0|),  &\qquad\; \; \; \; \; \; \; \; \mbox{if}\ \ x\in B(x_0,D)\setminus B(x_0,\frac{D}{2}),\\
&d,  & \mbox{if}\ \ x\in B(x_0,\frac{D}{2}),\\
\end{array}
\right.
\end{equation*}
where $|.|$ denotes the Euclidean norm in $\mathbb{R}^N.$ It is then easy to see that
\begin{equation*}
\Delta v_d=\left\{
\begin{array}{ccc}
\qquad 0,  \qquad \qquad  \qquad \mbox{if} \ \ x\in\Omega\setminus B(x_0,D)\cup B(x_0,\frac{D}{2}),\\
\\[-3mm]
\displaystyle \frac{-2d(N-1)}{D(x-x_0)}, \quad \quad \mbox{if} \ \ x\in B(x_0,D)\setminus B(x_0,\frac{D}{2}).
\end{array}
\right.
\end{equation*}

Using Lemma \ref{2.2} and the continuity of the embedding $L^{p^+}(\Omega)\hookrightarrow L^{p(x)}(\Omega)$, we can conclude that
\begin{eqnarray*}
\hspace{-1cm}\frac{1}{p^+}\Big[ \frac{2d(N-1)}{D^2}\Big]_{p}L\leq \Phi(v_d)\leq c_3 L^{\frac{1}{p^+}} |\alpha(x)|_{\frac{p(x)}{p(x)-1}} \frac{4d(N-1)}{D^2}
+c_3 \Big[\frac{4d(N-1)}{D^2}\Big]^pL,
\end{eqnarray*}

$$\Psi(v_d)\geq\displaystyle\int_{B(x_0,\frac{D}{2})}\frac{V(x)}{q(x)}|v_d|^{q(x)} dx\geq \frac{1}{q^+}v_0[d]_{q} m(\frac{D}{2})^N,$$
and hence $$\frac{\Psi(v_d)}{\Phi(v_d)}\geq\frac{\frac{1}{q^+}v_0[d]_{q} w(\frac{D}{2})^N}{c_3 L^{\frac{1}{p^+}} |\alpha(x)|_{\frac{p(x)}{p(x)-1}} \frac{4d(N-1)}{D^2}
+c_3 \Big[\frac{4d(N-1)}{D^2}\Big]^pL}=\gamma_d. $$

Next, from $r<\frac{1}{p^+}\Big[\frac{2d(N-1)}{D^2}\Big]_{p}L$, we get $r<\Phi(v_d)$.
Now, for each $u\in\Phi^{-1}((-\infty,r])$, due to condition ${\bf{(H_4)}}$, one has that
\begin{equation}\label{r1}
\frac{1}{p^+}\big[\|u\|\big]_p\leq r.
\end{equation}

Proposition \ref{pr} and inequalities \eqref{r1} and \eqref{k} now yield
\begin{eqnarray}\label{psi}
\Psi(u)&\leq&\frac{1}{q^{-}}|V|_{s(x)}||u|^{q(x)}|_{s'(x)}
\leq\frac{1}{q^-}|V|_{s(x)}\big[k\|u\|\big]^q \nonumber\\
&\leq&\frac{1}{q^-}|V|_{s(x)}[k]^q\big[(p^+)^{\frac{1}{p^-}}[r]^{\frac{1}{p}}\big]^q
\leq \frac{(p^+)^{\frac{q^+}{p^-}}}{q^-}[k]^q|V|_{s(x)} \big[[r]^{\frac{1}{p}}\big]^q.
\end{eqnarray}

Therefore  $$\frac{1}{r}\displaystyle\sup_{\Phi(u)\leq r}\Psi(u)\leq \overline{w}_r.$$

In the next step, we shall prove that for each $\lambda>0$, the energy functional $\Phi-\lambda\Psi$ is coercive.
By Remark \ref{embedding}, we have
\begin{eqnarray}\label{eqqq}
\Psi(u)&\leq& \frac{1}{q^-}\int_{\Omega}V(x)|u|^{q(x)}dx
\leq \frac{1}{q^-}|V|_{s(x)}\Big[k\|u\|\Big]^{q}.
\end{eqnarray}

For $\|u\|>1,$  relations \eqref{coe} and \eqref{eqqq} give the following
$$\Phi(u)-\lambda\Psi(u)\geq \frac{1}{p^+} \|u\|^{p^-}-\lambda \frac{1}{q^-}|V|_{s(x)}\Big[k\|u\|\Big]^{q}.$$

Since $1\leq q^-\leq q^+<p^-$, it follows that $\Phi(u)-\lambda\Psi(u)$ is coercive.
Finally, due to the fact that $$\overline{\Lambda}:=\Big(\frac{1}{\gamma_d}, \frac{1}{\overline{w}_r}\Big)\subseteq\Big( \frac{\Phi(v_d)}{\Psi(v_d)},\frac{r}{\sup_{\Phi(u)\leq r}\Psi(u)}\Big),$$
Theorem \ref{bonano} implies that for each $\lambda\in \overline{\Lambda}_r$, the functional  $\Phi-\lambda\Psi$ admits at least three critical points in X which are weak solutions for problem \eqref{P}. This completes the proof of Theorem \ref{principal}.
\hfill $\Box$

\section*{Acknowledgements}

The authors gratefully acknowledge comments and suggestions by the referee.  The third author acknowledges the support of the Deanship of Science Research (DSR) at King Abdulaziz University, Jeddah. The fourth author was supported by the Slovenian Research Agency program P1-0292 and grants N1-0114 and N1-0083.

\vskip 1.0cm

$^a$ (Khaled Kefi) Faculty of Computer Science and Information Technology, Northern Border University, Rafha, Kingdom of Saudi Arabia \& Mathematics Department, Faculty of Sciences, University of Tunis El-Manar, 1060 Tunis, Tunisia.\\
email:  {\it khaled\_kefi@yahoo.fr}\\  
https://orcid.org/0000-0001-9277-5820\\

$^b$ (Nawal Irzi) Mathematics Department, Faculty of Sciences, University of Tunis El-Manar, 1060 Tunis, Tunisia.\\
email: {\it nawal.irzi@fst.utm.tn}\\ 
https://orcid.org/0000-0002-7094-6278\\

$^c$ (Mohammed Mosa Al-Shomrani) Department of Mathematics, Faculty of Science, King Abdulaziz University, P.O. Box 80203, 21589 Jeddah, Saudi Arabia.\\
email: {\it malshomrani@hotmail.com}\\ 
https://orcid.org/0000-0001-7397-5165 \\

$^{d,e}$ (Du\v{s}an D. Repov\v{s}) Faculty of Education and Faculty of Mathematics and Physics, University of Ljubljana \&
Institute of Mathematics, Physics and Mechanics, 1000 Ljubljana, Slovenia.\\
email: {\it dusan.repovs@guest.arnes.si}\\ 
https://orcid.org/0000-0002-6643-1271
\end{document}